\pdfoutput=1
\RequirePackage{ifpdf}
\ifpdf 
\documentclass[pdftex]{sigma}
\else
\documentclass{sigma}
\fi

\numberwithin{equation}{section}

\newtheorem{Theorem}{Theorem}[section]
\newtheorem{Lemma}[Theorem]{Lemma}

\newcommand{\R}{\mathbb{R}}
\newcommand{\C}{\mathbb{C}}
\newcommand{\Z}{\mathbb{Z}}
\newcommand{\N}{\mathbb{N}}

\renewcommand{\H}{H^{(1)}}
\newcommand{\HH}{H^{(2)}}
\newcommand{\w}{w^{(1)}}
\newcommand{\ww}{w^{(2)}}
\DeclareMathOperator{\sn}{sn} 
\DeclareMathOperator{\cn}{cn} 
\DeclareMathOperator{\dn}{dn} 
\DeclareMathOperator{\am}{am} 

\begin{document}

\newcommand{\arXivNumber}{1808.04877}

\renewcommand{\PaperNumber}{131}

\FirstPageHeading

\ShortArticleName{Eigenvalue Problems for Lam\'e's Differential Equation}

\ArticleName{Eigenvalue Problems for Lam\'e's Differential Equation}

\Author{Hans VOLKMER}

\AuthorNameForHeading{H.~Volkmer}

\Address{Department of Mathematical Sciences, University of Wisconsin-Milwaukee,\\ P.O.~Box 413, Milwaukee, WI, 53201, USA}
\Email{\href{mailto:volkmer@uwm.edu}{volkmer@uwm.edu}}

\ArticleDates{Received August 14, 2018, in final form December 06, 2018; Published online December 12, 2018}

\Abstract{The Floquet eigenvalue problem and a generalized form of the Wangerin eigenvalue problem for Lam\'e's differential equation are discussed. Results include comparison theorems for eigenvalues and analytic continuation, zeros and limiting cases of (generalized) Lam\'e--Wangerin eigenfunctions. Algebraic Lam\'e functions and Lam\'e polynomials appear as special cases of Lam\'e--Wangerin functions.}


\Classification{33E10; 34B30}

\section{Introduction}
The Lam\'e equation (Arscott \cite[Chapter~IX]{A}) is
\begin{gather}\label{I:lame}
\frac{{\rm d}^2w}{{\rm d}z^2}+\big(h-\nu(\nu+1)k^2 \sn^2(z,k)\big) w=0 ,
\end{gather}
where $\sn(z,k)$ is the Jacobian elliptic function with modulus $k\in(0,1)$ (Whittaker and Watson \cite[Chapter~XXII]{WW}), $\nu\in \R$ and $h$ is the eigenvalue parameter. This equation has regular singularities at the points $z=2mK+{\rm i}(2n+1)K'$, where $m$, $n$ are integers and $K=K(k)$ and $K'=K'(k)$ denote complete elliptic integrals. Various eigenvalue problems for the Lam\'e equation have been treated in the literature.

The Lam\'e equation is an even Hill's equation with fundamental period $2K$. The theory of Hill's equation is well-known; see, e.g., Arscott~\cite{A}, Eastham~\cite{Eastham} and Magnus and Winkler~\cite{MW}. Results on the periodic eigenvalue problem specific for the Lam\'e equation with eigenfunctions satisfying $w(z+2K)=\pm w(z)$ can be found in \cite[Section~15,5]{EMO}, \cite{I1,J} and \cite[Chapter~29]{O}. These functions have many applications; see, e.g.,~\cite{BMT}. In Section \ref{F} of this paper we will consider the more general Floquet eigenvalue problem $w(z+2K)={\rm e}^{{\rm i}\mu\pi} w(z)$. For the general Hill's equation this eigenvalue problem is treated in Eastham~\cite{Eastham}. Some results on the Floquet eigenvalue problem specific for the Lam\'e equation can be found in Ince~\cite[Sections~7 and~8]{I2}.

Wangerin \cite{W} showed that Lam\'e's equation appears when Laplace's equation is separated in confocal cyclidic coordinates of revolution. Such coordinate systems can be found in Moon and Spencer~\cite{MSP} and in Miller~\cite{Miller}. They include flat-ring, flat-disk, bi-cyclide and cap-cyclide coordinates. An outline of Wangerin's results is given in \cite[Section~15.1.3]{EMO}. In order to obtain harmonic functions relevant for applications special solutions of the Lam\'e equation called Lam\'e--Wangerin functions were introduced; see Erd\'elyi~\cite{Erd2}, \cite[p.~88]{EMO}. The Lam\'e--Wangerin eigenvalue problem is obtained when we require
that $(\sn z)^{1/2} w(z)$ stays bounded at the singularities ${\rm i}K'$ and $2K+{\rm i}K'$; see Erd\'elyi \cite{Erd2} and Erd\'elyi, Magnus and Oberhettinger~\cite[Section~15.6]{EMO}. These eigenfunctions will be defined on the segment $({\rm i}K',2K+{\rm i}K')$ but can then be continued analytically.

In Section~\ref{W} of this paper we consider a more general eigenvalue problem whose eigenfunc\-tions~$w(z)$ have the form
\begin{gather*} w(z)=(z-{\rm i}K')^{\nu+1}\sum_{n=0}^\infty q_n (z-{\rm i}K')^{2n} \end{gather*}
at $z={\rm i}K'$ and a similar condition at $z=2K+{\rm i}K'$. We call these eigenfunctions generalized Lam\'e--Wangerin functions. Every classical Lam\'e--Wangerin function is also a generalized Lam\'e--Wangerin function but not vice versa unless $\nu\ge -\frac12$.

The motivation for introducing these functions is as follows. In Section~\ref{F} we show that the eigenvalues of the Floquet eigenvalue problem agree with the eigenvalues of an infinite tridiagonal matrix $F$ (considered in the Hilbert sequence space $\ell^2(\Z)$, $\Z$ the set of integers). One is especially interested in the case that a matrix entry in the diagonal above or below the main diagonal of $F$ vanishes because then the eigenvalue problem splits in two problems whose eigenvalues are given by infinite tridiagonal submatrices of $F$ that are only infinite in one direction (the underlying Hilbert space can be taken as $\ell^2(\N_0)$, $\N_0$ the set of non-negative integers). It turns out that this special case occurs if and only if $\nu+\mu$ or $\nu-\mu$ is an integer ($\nu$ is the parameter in Lam\'e's equation and $\mu$ is the parameter in Floquet's condition.) Interestingly, if $\nu+\mu$ or $\nu-\mu$ is an integer then the eigenvalues of one of the submatrices are identical with the eigenvalues of a~classical Lam\'e--Wangerin problem. This is a simple observation but as far as I know has not been stated in the literature. Of course, the obvious question is: If the eigenvalues of one submatrix of~$F$ are those for a classical Lam\'e--Wangerin problem what is the meaning of the eigenvalues of the complementary submatrix? As we show in this paper, these are the eigenvalues for a~generalized Lam\'e--Wangerin problem (non-classical except when $\nu=-\frac12$).

A second motivation for introducing the generalized Lam\'e--Wangerin functions is as follows. Lam\'e polynomials and algebraic Lam\'e functions are not special cases of classical Lam\'e--Wangerin functions. However, they are special case of generalized Lam\'e--Wangerin function. We show in Sections~\ref{AL} and~\ref{LP} how Lam\'e polynomials and algebraic Lam\'e functions appear in the notation of generalized Lam\'e--Wangerin functions. We should mention that we adopt the name ``algebraic Lam\'e functions'' from \cite[p.~68]{EMO}. These functions are called ``Lam\'e--Wangerin functions'' in Lambe~\cite{L} and non-meromorphic Lam\'e functions in Finkel et al.~\cite{F}.

In Section \ref{C} we compare the eigenvalues of the Floquet and the Lam\'e--Wangerin problems. In Sections~\ref{AL} and~\ref{LP} we show that algebraic Lam\'e functions and Lam\'e polynomials are special cases of (generalized) Lam\'e--Wangerin functions. In Section~\ref{Z} we investigate the number of zeros of Lam\'e--Wangerin eigenfunctions. In Section~\ref{L} we find the limit of Lam\'e--Wangerin functions as $k\to 0$.

I consider some of the results in this paper as new but not all results are new. The treatment of the generalized Lam\'e--Wangerin problem is new. The recursions \eqref{W:rec1} and \eqref{W:rec3} are known from \cite{EMO} but the ``symmetric'' recursions \eqref{W:rec2}, \eqref{W:rec4} appear to be new. The latter recursions are used in some proofs and also in Section~\ref{AL}. The results in Sections~\ref{AC}, \ref{C}, \ref{Z} and~\ref{L} are new. Lam\'e polynomials and algebraic Lam\'e functions are well-known, so I make no claim that Sections~\ref{AL} and~\ref{LP} contain new results.

\section{Floquet solutions}\label{F}
On the real axis $z\in\R$, \eqref{I:lame} is a Hill equation with fundamental period $2K$. Let $\mu\in\R$. We call $h$ a Floquet eigenvalue if there exists a nontrivial solution~$w$ of~\eqref{I:lame} satisfying
\begin{gather}\label{F:Floquet}
w(z+2K)={\rm e}^{{\rm i}\pi\mu} w(z),\qquad z\in \R.
\end{gather}
$w(z)$ is a corresponding Floquet eigenfunction. It is known \cite[p.~31]{Eastham} that the eigenvalues are real and form a sequence converging to $\infty$. We denote the eigenvalues by
\begin{gather*}
 h_0(\mu,\nu,k)\le h_1(\mu,\nu,k)\le h_2(\mu,\nu,k)\le \cdots.
 \end{gather*}
The eigenvalues are counted according to multiplicity.
If $\mu$ is not an integer then
\begin{gather*}
 h_0(\mu,\nu,k)<h_1(\mu,\nu,k)< h_2(\mu,\nu,k)< \cdots.
\end{gather*}
Obviously, we have
\begin{gather}\label{F:rules}
h_m(\mu,\nu,k)=h_m(\mu+2,\nu,k)=h_m(-\mu,\nu,k)=h_m(\mu,-\nu-1,k) .
\end{gather}

Let $w_1(z,h,\nu,k)$ and $w_2(z,h,\nu,k)$ be the solutions of \eqref{I:lame} satisfying the initial conditions $w_1(0)=1$, $\frac{{\rm d}w_1}{{\rm d}z}(0)=0$, $w_2(0)=0$, $\frac{{\rm d}w_2}{{\rm d}z}(0)=1$. Then Hill's discriminant $D$ is given by
\begin{gather*} D(h,\nu,k)=w_1(2K,h,\nu,k)+\frac{{\rm d}w_2}{{\rm d}z}(2K,h,\nu,k) .\end{gather*}
The eigenvalues $h_m(\mu,\nu,k)$ are the solutions of the equation
\begin{gather}\label{F:disc}
 D(h,\nu,k)=2\cos(\mu\pi) ;
\end{gather}
see \cite[equation~(2.4.4)]{Eastham}. From \eqref{F:disc} we easily obtain the following result that will be needed
later.

\begin{Theorem}\label{F:t1}For every $m\in\N_0=\{0,1,2,\dots\}$, the function $(\mu,\nu,k)\mapsto h_m(\mu,\nu,k)$ is con\-ti\-nuous on $\R\times\R\times[0,1)$.
\end{Theorem}

Theorem \ref{F:t1} can also be inferred from results on Sturm--Liouville theory~\cite{KWZ}.

Following \cite[p.~65]{EMO} we transform \eqref{I:lame} by setting
\begin{gather}\label{F:am}
 t=\tfrac12\pi-\am(z,k),
\end{gather}
where $\am$ is Jacobi's amplitude function. We note that \eqref{F:am} establishes a conformal mapping between the strip $|\Im z|<K'$ and the $t$-plane cut along the rays $m\pi\pm {\rm i} sL$, $s\ge 1$, $m\in\Z$, where
\begin{gather*} L:=\operatorname{arccosh} \frac1k=\frac12\ln \frac{1+k'}{1-k'},\qquad k'=\sqrt{1-k^2}. \end{gather*}
Then
\begin{gather*} \sn z =\cos t, \qquad \cn z =\sin t .\end{gather*}
We obtain
\begin{gather}\label{F:lame2}
\big(1-k^2\cos^2 t\big)\frac{{\rm d}^2w}{{\rm d}t^2}+k^2\cos t \sin t\frac{{\rm d}w}{{\rm d}t}+\big(h-\nu(\nu+1)k^2\cos^2 t\big) w =0.
\end{gather}
Since $\am(z+2K)=\am z+\pi$, condition \eqref{F:Floquet} becomes
\begin{gather*}
 w(t+\pi)={\rm e}^{-{\rm i}\pi \mu} w(t),\qquad t\in\R .
\end{gather*}
This condition is equivalent to ${\rm e}^{{\rm i}\mu t}w(t)$ being periodic with period~$\pi$. Therefore, using Fourier series, eigenfunctions have the form
\begin{gather}\label{F:expansion1}
 w(t)=\sum_{n=-\infty}^\infty c_n {\rm e}^{-{\rm i}(\mu+ 2n) t}.
\end{gather}
By substituting \eqref{F:expansion1} in \eqref{F:lame2}, we obtain the three-term recursion
\begin{gather}\label{F:rec1}
\rho_n c_{n-1}+(\sigma_n-h) c_n+\tau_{n+1}c_{n+1}=0,\qquad n\in\Z,
\end{gather}
where
\begin{gather*}
\rho_n= -\tfrac14k^2 (2n-1+\mu+\nu)(2n-2+\mu-\nu),\\
\sigma_n=\tfrac12k^2 \nu(\nu+1)+\big(1-\tfrac12 k^2\big)(2n+\mu)^2,\\
\tau_n=-\tfrac14k^2(2n+\mu+\nu)(2n-1+\mu-\nu).
\end{gather*}
This recursion is similar to the one given in \cite[equation~(7.1)]{I2} which is based on Fourier cosine series instead of the complex form of Fourier series we used. The behavior of solutions $\{c_n\}_{n\in\Z}$ of~\eqref{F:rec1} as $n\to\infty$ is given by Perron's rule~\cite{P}. If $k\in(0,1)$ we choose $n_0$ so large that $\rho_n\ne 0$ and $\tau_{n+1}\ne 0$ for $n\ge n_0$. Then the solutions $\{c_n\}_{n>n_0}$ of equations~\eqref{F:rec1} for $n\ge n_0$ form a~two-dimensional vector space. There exists a recessive solution which is uniquely determined up to a constant factor with the property
\begin{gather}\label{F:recessive1}
 \lim_{n\to\infty} \frac{c_{n+1}}{c_n}=\frac{1-k'}{1+k'}<1.
\end{gather}
Every solution which is linearly independent of this solution satisfies
\begin{gather*}
 \lim_{n\to\infty} \frac{c_{n+1}}{c_n}=\frac{1+k'}{1-k'}>1.
\end{gather*}
Similar results hold for $n\to-\infty$. We obtain the following theorem.

\begin{Theorem}\label{F:t}Let $\mu,\nu\in\R$ and $k\in(0,1)$. Then $h$ is one of the eigenvalues $h_m(\mu,\nu,k)$
if and only if the recursion \eqref{F:rec1} has a nontrivial solution $\{c_n\}_{n\in\Z}$ such that
\begin{enumerate}\itemsep=0pt
\item[$a)$] either there is $n_0$ such that $c_n=0$ for $n\ge n_0$ or $\{c_n\}$ is recessive as $n\to\infty$; and
\item[$b)$] either there is $n_0$ such that $c_n=0$ for $n\le n_0$ or $\{c_n\}$ is recessive as $n\to-\infty$.
\end{enumerate}
The expansion \eqref{F:expansion1} of a corresponding eigenfunction converges in the strip $|\Im t|<L$.
\end{Theorem}

Of course, a nontrivial solution $\{c_n\}$ of \eqref{F:rec1} can be zero for $n\ge n_0$ or $n\le n_0$ only when one of the numbers $\rho_n$ or $\tau_n$ vanishes. This happens if and only if at least one of the numbers $\mu\pm\nu$ is an integer. These interesting cases will be discussed in Sections~\ref{C}, \ref{AL} and \ref{LP}.

Alternatively, we may expand
\begin{gather*}
w(t)=\big(1-k^2\cos^2 t\big)^{1/2}\sum_{n=-\infty}^\infty d_n {\rm e}^{-{\rm i}(\mu+ 2n) t}.
\end{gather*}
Then we obtain the ``adjoint'' recursion
\begin{gather}\label{F:rec2}
\tau_n d_{n-1}+(\sigma_n-h) d_n+\rho_{n+1}d_{n+1}=0,\qquad n\in\Z.
\end{gather}

Theorem \ref{F:t} also holds with \eqref{F:rec2} in place of \eqref{F:rec1}.

\section{Generalized Lam\'e--Wangerin functions}\label{W}

A (classical) Lam\'e--Wangerin function $w(z)$ is a nontrivial solution of Lam\'e's equation \eqref{I:lame} with the property that $(\sn z)^{1/2}w(z)$ stays bounded on the segment between the regular singularities~$K'$ and $2K+{\rm i}K'$; see \cite[Section~15.6]{EMO}. Such solutions exist only for specific values of $h$.
If we substitute $z=u+{\rm i}K'$ then we obtain the singular Sturm--Liouville problem \cite{Z}
\begin{gather}\label{W:SL}
\frac{{\rm d}^2w}{{\rm d}u^2}+\big(h-\nu(\nu+1)\sn^{-2}(u,k)\big) w =0,\qquad 0<u<2K,
\end{gather}
with the boundary condition that $(\sn u)^{-1/2} w$ stays bounded on $(0,2K)$.

The eigenvalue problem splits into two problems, one for functions that are even with respect to $K+{\rm i}K'$, that is,
\begin{gather}\label{W:even}
 w(K+{\rm i}K'+s)=w(K+{\rm i}K'-s) \qquad \text{for $-K<s<K$},
\end{gather}
and one for functions which are odd with respect to $K+{\rm i}K'$, that is,
\begin{gather}\label{W:odd}
w(K+{\rm i}K'+s)=-w(K+{\rm i}K'-s) \qquad \text{for $-K<s<K$}.
\end{gather}
Without loss of generality, one may assume that $\nu\ge -\frac12$, and since the exponents at ${\rm i}K'$ and $2K+{\rm i}K'$ are $\{\nu+1, -\nu\}$, a Lam\'e--Wangerin function has the form
\begin{gather}\label{W:form}
 w(z)=(z-{\rm i}K')^{\nu+1}\sum_{n=0}^\infty q_n (z-{\rm i}K')^{2n}
\end{gather}
for $z$ close to ${\rm i}K'$ with $q_0\ne 0$.

We generalize these eigenvalue problems as follows. Let $\nu\in\R$, $0<k<1$. We call $h\in\C$ an eigenvalue of the first Lam\'e--Wangerin problem if~\eqref{I:lame} admits a nontrivial solution~$w$ on the interval $({\rm i}K',2K+{\rm i}K')$ which close to $z={\rm i}K'$ has the form~\eqref{W:form} and satisfies $w'(K+{\rm i}K')=0$. The latter property is equivalent to~\eqref{W:even}. The eigenfunction $w$ will be called a Lam\'e--Wangerin function of the first kind. Note that we consider this eigenvalue problem for all real $\nu$ not just for $\nu\ge -\frac12$. Also note that the condition $q_0\ne 0$ is not required in~\eqref{W:form} although $q_0\ne 0$ will hold if $\nu+\frac12$ is not a negative integer.

Similarly, we call $h$ an eigenvalue of the second Lam\'e--Wangerin problem if~\eqref{I:lame} admits a~nontrivial solution~$w$ on the interval $({\rm i}K',2K+{\rm i}K')$ which close to $z={\rm i}K'$ has the form~\eqref{W:form}
and satisfies $w(K+{\rm i}K')=0$. The latter property is equivalent to~\eqref{W:odd}. The eigenfunction~$w$ will be called a Lam\'e--Wangerin function of the second kind.

If $\nu>-\frac32$ our eigenvalue problems are included in singular Sturm--Liouville theory (see also~\cite{MN}) but this theory does not give us results for $\nu\le -\frac32$. We will treat these eigenvalue problems by a~different method developed below.

We substitute
\begin{gather}\label{W:etat}
 \eta = {\rm e}^{-2{\rm i}t}
\end{gather}
in \eqref{F:lame2}. We obtain the Fuchsian equation
\begin{gather}
k^2\eta(\eta-\eta_1)(\eta-\eta_2)\left[\frac{{\rm d}^2w}{{\rm d}\eta^2}+\frac12\left(\frac{1}{\eta}+\frac{1}{\eta-\eta_1}+\frac{1}{\eta-\eta_2}\right)\frac{{\rm d}w}{{\rm d}\eta}\right]\nonumber\\
\qquad{} +\left(h-k^2\nu(\nu+1)\frac{(1+\eta)^2}{4\eta}\right) w =0,\label{W:lame1}
\end{gather}
where
\begin{gather*} \eta_1:=\frac{1-k'}{1+k'}\in(0,1),\qquad \eta_2:=\frac{1+k'}{1-k'}\in(1,\infty) .\end{gather*}
The differential equation~\eqref{W:lame1} has regular singularities at $\eta=0,\eta_1,\eta_2,\infty$ with exponents $\big\{{-}\frac12\nu,\frac12(\nu+1)\big\}$, $\big\{0,\frac12\big\}$, $\big\{0,\frac12\big\}$, $\big\{{-}\frac12\nu,\frac12(\nu+1)\big\}$, respectively. If we combine \eqref{F:am} with \eqref{W:etat} we obtain
\begin{gather}\label{W:etaz}
 \eta={\rm e}^{-2{\rm i}(\frac12\pi-\am z)}=(\sn z+{\rm i}\cn z)^{-2}=(\sn z-{\rm i}\cn z)^2.
\end{gather}
Setting $z=u+{\rm i}K'$ for $0<u<K$ this gives
\begin{gather*} \eta=\frac{1-\dn u}{1+\dn u}.\end{gather*}
This establishes a bijective increasing map between $u\in(0,K)$ and $\eta\in (0,\eta_1)$. Taking into consideration the behavior of $\eta$ close to $u=0$ and $u=K$ we see that a~Lam\'e--Wangerin function of the first kind expressed in the variable $\eta$ is a solution of~\eqref{W:lame1} on $(0,\eta_1)$ which close to $\eta=0$ is of the form
\begin{gather}\label{W:expansion1}
 w(\eta)=\eta^{(\nu+1)/2}\sum_{n=0}^\infty c_n\eta^n,
\end{gather}
and which is analytic at $\eta=\eta_1$. This implies that the radius of convergence of the power series in~\eqref{W:expansion1} is $\ge \eta_2$. For the coefficients $c_n$ we find the recursion
\begin{gather}
 \big(\beta^{(1)}_0-h\big)c_0+\gamma_1 c_1 =0,\nonumber \\
 \alpha_n c_{n-1}+\big(\beta_n^{(1)}-h\big) c_n+ \gamma_{n+1}c_{n+1}=0,\qquad n\ge 1,\label{W:rec1}
\end{gather}
where
\begin{gather*}
\alpha_n = -\tfrac12k^2(n+\nu)(2n-1),\\
\beta_n^{(1)} = \tfrac12k^2\nu(\nu+1)+\big(1-\tfrac12k^2\big)(2n+\nu+1)^2,\\
\gamma_n = -\tfrac12k^2(2n+2\nu+1)n .
\end{gather*}
Note that the equations \eqref{W:rec1} for $n\ge 1$ agree with~\eqref{F:rec1} when we set $\mu=\nu+1$. The recursion~\eqref{W:rec1} is given in~\cite[Section~15.6(15)]{EMO}.

Using Perron's rule, we see that $h$ is an eigenvalue of the first Lam\'e--Wangerin problem if and only if~\eqref{W:rec1} has a nontrivial solution $\{c_n\}_{n=0}^\infty$ which is either identically zero for large~$n$ or satisfies~\eqref{F:recessive1}. Of course, a nontrivial solution $\{c_n\}_{n=0}^\infty$ of \eqref{W:rec1} can be identically zero for large~$n$ only if one of the numbers~$\alpha_n$ is zero, that is, if $\nu$ is a negative integer.

Alternatively, we may expand a Lam\'e--Wangerin function of the first kind in the form
\begin{gather}\label{W:expansion2}
w(\eta)=\eta^{(\nu+1)/2}(\eta_2-\eta)^{1/2}\sum_{n=0}^\infty a_n\eta^n
\end{gather}
with the power series having radius $\ge \eta_2$. In order to find the recursion for the coefficients $a_n$ we transform~\eqref{W:lame1} by setting
\begin{gather*} w(\eta)=(\eta_2-\eta)^{1/2}v(\eta)\end{gather*}
to
\begin{gather}
k^2\eta(\eta-\eta_1)(\eta-\eta_2)\left[\frac{{\rm d}^2v}{{\rm d}\eta^2} +\frac12\left(\frac{1}{\eta}+\frac{1}{\eta-\eta_1}+\frac{3}{\eta-\eta_2}\right)\frac{{\rm d}v}{{\rm d}\eta}\right]\nonumber\\
\qquad{} +\left(h-k^2\nu(\nu+1)\frac{(1+\eta)^2}{4\eta}+\tfrac14 k^2(2\eta-\eta_1)\right) v =0.\label{W:lame2}
\end{gather}
We obtain the recursion
\begin{gather}
 \big(\epsilon^{(1)}_0-h\big)a_0+\delta_1 a_1 =0,\nonumber\\
 \delta_n a_{n-1}+\big(\epsilon_n^{(1)}-h\big) a_n+ \delta_{n+1}a_{n+1}=0,\qquad n\ge 1,\label{W:rec2}
\end{gather}
where
\begin{gather*}
\delta_n= -\tfrac12k^2 n(2n+2\nu+1),\\
\epsilon_n^{(1)}= \tfrac12k^2\nu(\nu+1)+\big(1-\tfrac12k^2\big)(2n+\nu+1)^2+\tfrac14 k^2\eta_1(4n+2\nu+3)\\
\hphantom{\epsilon_n^{(1)}}{} =\tfrac12 k^2\nu(\nu+1)-k'\big(2n+\tfrac32+\nu\big)+\big(1-\tfrac12k^2\big)\big(\tfrac14+\big(2n+\tfrac32+\nu\big)^2\big).
\end{gather*}

It is a pleasant surprise that, in contrast to \eqref{W:rec1}, recursion \eqref{W:rec2} is of self-adjoint form. We take advantage of this observation and introduce a symmetric operator $S=S^{(1)}(\nu,k)$ in the Hilbert space $\ell^2(\N_0)$ with the standard inner product. The domain of definition of $S$ is
\begin{gather*} D(S)=\left\{ \{x_n\}_{n=0}^\infty\colon \sum_{n=0}^\infty n^4|x_n|^2 <\infty\right\} \end{gather*}
and $S$ is defined on $D(S)$ by
\begin{gather*}
 S(\{x_j\})_0=\epsilon_0^{(1)} x_0 +\delta_1 x_1,\\
S(\{x_j\})_n=\delta_n x_{n-1} +\epsilon_n^{(1)} x_n +\delta_{n+1}x_{n+1} ,\qquad n\ge 1.
\end{gather*}
So $S$ is represented by an infinite symmetric tridiagonal matrix.

\begin{Theorem}\label{W:t1}Let $\nu\in\R$ and $k\in[0,1)$.
\begin{enumerate}\itemsep=0pt
\item[$(a)$] $S^{(1)}(\nu,k)$ is a self-adjoint operator in $\ell^2(\N_0)$ with compact resolvent and bounded below.
\item[$(b)$] If $k\in(0,1)$ the eigenvalues of $S^{(1)}(\nu,k)$ agree with the eigenvalues of the first Lam\'e--Wangerin problem.
\item[$(c)$] If $k\in(0,1)$ the eigenvalues of $S^{(1)}(\nu,k)$ are simple.
\end{enumerate}
\end{Theorem}
\begin{proof} (a) We abbreviate $S=S^{(1)}(\nu,k)$, and write $S=A+B$ with $A=S^{(1)}(\nu,0)$. So $A$ is represented by an infinite diagonal matrix with diagonal entries $(2n+\nu+1)^2$, $n\in\N_0$. It is clear that $A$ is a positive semi-definite self-adjoint operator with compact resolvent. There are two constants $\lambda>0$ and $c\in(0,1)$ such that
\begin{gather}\label{3:eq1}
 \|B x\|\le c\| (A+\lambda)x\|\qquad\text{for all $x\in D(S)$} .
\end{gather}
To prove this it is convenient to write $B=B_1+B_2+B_3$ where each $B_i$ has a matrix representation consisting of only one nonzero ``diagonal'', and estimate $\|Bx\|\le \|B_1x\|+\|B_2x\|+\|B_3x\|$. We can reach $c<1$ because the factor of $n^2$ on the main diagonal of $A$ is $4$ while the factors of~$n^2$ on the three diagonals of~$B$ are $-k^2$, $-2k^2$, $-k^2$, respectively. From~\eqref{3:eq1} we obtain that $T:=B(A+\lambda)^{-1}$ is a bounded linear operator with operator norm $\|T\|\le c<1$. Therefore, $1+T$ is invertible and
\begin{gather*} (S+\lambda)^{-1}=(A+\lambda+B)^{-1} =(A+\lambda)^{-1}(1+T)^{-1} .\end{gather*}
This shows that $(S+\lambda)^{-1}$ is a compact operator. Since $S$ is symmetric, we find that $S$ is self-adjoint; compare \cite[Chapter~V, Theorem~4.3]{K}. From \eqref{3:eq1} we also obtain that $S+\lambda$ is positive definite \cite[Chapter~V, Theorem~4.11]{K}. Therefore, (a) follows.

(b) $h$ is an eigenvalue of $S$ if and only if the recursion \eqref{W:rec2} has a nontrivial solution $\{a_n\}_{n=0}^\infty$ with the property that $\sum\limits_{n=0}^\infty n^4 |a_n|^2<\infty$. By Perron's rule the latter property is equivalent to $a_n=0$ for large $n$ or $\{a_n\}$ is recessive as $n\to\infty$.

(c) If $k\in(0,1)$ the eigenvalues of $S$ are simple because the corresponding eigenfunctions of the first Lam\'e--Wangerin problem are even with respect to $K+{\rm i}K'$.
\end{proof}

Based on Theorem~\ref{W:t1} we write the eigenvalues of the first Lam\'e--Wangerin problem with $k\in(0,1)$ in the form
\begin{gather*} \H_0(\nu,k)<\H_1(\nu,k)<\H_2(\nu,k)<\cdots. \end{gather*}
The Lam\'e--Wangerin function belonging to $\H_m(\nu,k)$ will be denoted by $w_m^{(1)}(z,\nu,k)$. If a normalization is required it will be stated separately. We note that the corresponding eigenvectors $\{a_n\}_{n=0}^\infty$ of $S$ when properly normalized form an orthonormal basis in the Hilbert space~$\ell^2(\N_0)$.

The eigenvalues of $S^{(1)}(\nu,0)$ are $(2n+\nu+1)^2$ for $n\in\N_0$. If we arrange this sequence in increasing order repeated according to multiplicity we denote these eigenvalues by $\H_m(\nu,0)$. Explicitly, they are given by the following lemma.

\begin{Lemma}\label{W:l1} Let $p-1< \nu\le p$ with $p\in\Z$. Then, for all $m\in\N_0$,
\begin{gather*} \H_m(\nu,0)=(2\ell+\nu+1)^2,\end{gather*}
where
\begin{gather*} \ell=\begin{cases} m & \text{if $m+p\ge 0$},\\
\frac12(m-p) & \text{if $m+p<0$, $m+p$ even},\\
\frac12(-m-p-1) & \text{if $m+p<0$, $m+p$ odd}.\\
\end{cases}
\end{gather*}
\end{Lemma}

We will need continuity of the eigenvalues $\H_m(\nu,k)$.

\begin{Theorem}\label{W:t2} The function $(\nu,k)\mapsto \H_m(\nu,k)$ is continuous on $\R\times [0,1)$ for every $m\in\N_0$.
\end{Theorem}
\begin{proof}Set $S(\nu,k)=S^{(1)}(\nu,k)$, $A(\nu)=S(\nu,0)$ and $S(\nu,k)=A(\nu)+B(\nu,k)$. Let $\nu_0>0$ and $k_0\in(0,1)$ be given, and set $\Omega:=[-\nu_0,\nu_0]\times [0,k_0]$. Then we can find $\lambda>0$ large enough and $c\in(0,1)$ such that~\eqref{3:eq1} holds uniformly for $(\nu,k)\in\Omega$. It follows that $T(\nu,k):=B(\nu,k) (A(\nu)+\lambda)^{-1}$ is a bounded linear operator with operator norm $\|T(\nu,k)\|\le c$ for all $(\nu,k)\in\Omega$. As before, we have
\begin{gather}\label{3:eq2}
 (S(\nu,k)+\lambda)^{-1}=(A(\nu)+\lambda)^{-1} (1+T(\nu,k))^{-1} ,\qquad (\nu,k)\in\Omega.
\end{gather}
Suppose we have a sequence $(\nu_n, k_n)\in\Omega$ which converges to $\big(\hat\nu, \hat k\big)$ as $n\to\infty$. Then we can easily show using the definitions of $A$ and $T$ that
\begin{gather*} \big\|(A(\nu_n)+\lambda)^{-1}-(A(\hat\nu)+\lambda)^{-1}\big\|\to 0 \qquad \text{as $n\to\infty$} ,\end{gather*}
and
\begin{gather*} \big\|T(\nu_n,k_n)-T\big(\hat\nu,\hat k\big)\big\|\to 0 \qquad \text{as $n\to\infty$} .\end{gather*}
Using \eqref{3:eq2} we then obtain that
\begin{gather}\label{3:eq3}
\big\|(S(\nu_n,k_n)+\lambda)^{-1}-\big(S\big(\hat\nu,\hat k\big)+\lambda\big)^{-1}\big\|\to0\qquad \text{as $n\to \infty$.}
\end{gather}
If $K_n$ is a sequence of positive definite compact Hermitian operators converging to a positive definite compact Hermitian operator $K$ with respect to the operator norm, then the $m$th largest eigenvalue of $K$ (counted according to multiplicity) converges to the $m$ largest eigenvalue of $K$ as $n\to\infty$ for every $m\in\N_0$. This follows directly from the minimum-maximum-principle If we set $K_n=(S(\nu_n,k_n)+\lambda)^{-1}$, $K=\big(S\big(\hat\nu,\hat k\big)+\lambda\big)^{-1}$ and use~\eqref{3:eq3} we obtain $\H_m(\nu_n,k_n)\to \H_m\big(\hat \nu,\hat k\big)$ as $n\to\infty$ for every $m\in\N_0$ as desired.
\end{proof}

A Lam\'e--Wangerin function of the second kind can be written in the form
\begin{gather}\label{W:expansion3}
 w(\eta)=\eta^{(\nu+1)/2}(\eta_1-\eta)^{1/2}(\eta_2-\eta)^{1/2}\sum_{n=0}^\infty d_n \eta^n,
\end{gather}
where the power series $\sum d_n\eta^n$ has radius $\ge \eta_2$. If we set
\begin{gather*} w(\eta)=(\eta-\eta_1)^{1/2}(\eta-\eta_2)^{1/2} v(\eta) \end{gather*}
in \eqref{W:lame1}, we obtain
\begin{gather*}
 k^2\eta(\eta-\eta_1)(\eta-\eta_2)\left[\frac{{\rm d}^2v}{{\rm d}\eta^2}
+\frac12\left(\frac{1}{\eta}+\frac{3}{\eta-\eta_1}+\frac{3}{\eta-\eta_2}\right)\frac{{\rm d}v}{{\rm d}\eta}\right]\nonumber\\
\qquad{} +\left(\tfrac32 k^2\eta-1+\tfrac12k^2+h-k^2\nu(\nu+1)\frac{(1+\eta)^2}{4\eta}\right) v =0.
\end{gather*}
This gives the recursion
\begin{gather}
 \big(\beta_0^{(2)}-h\big)d_0+\gamma_1 d_1 =0,\nonumber\\
 \alpha_{n+1} d_{n-1}+\big(\beta_n^{(2)}-h\big) d_n+ \gamma_{n+1}d_{n+1}=0,\quad n\ge 1,\label{W:rec3}
\end{gather}
where $\alpha_n$, $\gamma_n$ are as in \eqref{W:rec1} and
\begin{gather*}
\beta_n^{(2)}= \tfrac12k^2\nu(\nu+1)+\big(1-\tfrac12k^2\big)(2n+\nu+2)^2.
\end{gather*}
Note that the equations \eqref{W:rec3} for $n\ge 1$ agree with \eqref{F:rec2} when we set $\mu=\nu+2$. The recursion~\eqref{W:rec3} is given in \cite[Section~15.6(16)]{EMO}.

Alternatively, a Lam\'e--Wangerin function of the second kind can be written as
\begin{gather}\label{W:expansion4}
 w(\eta)=\eta^{(\nu+1)/2}(\eta_1-\eta)^{1/2}\sum_{n=0}^\infty b_n \eta^n,
\end{gather}
where the power series $\sum b_n\eta^n$ has radius $\ge \eta_2$. If we set
\begin{gather*} w(\eta)=(\eta_1-\eta)^{1/2} v(\eta) \end{gather*}
in \eqref{W:lame1}, we obtain \eqref{W:lame2} with $\eta_1$, $\eta_2$ interchanged. This gives the recursion
\begin{gather}
\big(\epsilon_0^{(2)}-h\big)b_0+\delta_1 b_1 =0,\nonumber\\
 \delta_n b_{n-1}+\big(\epsilon_n^{(2)}-h\big) b_n+ \delta_{n+1}b_{n+1}=0,\qquad n\ge 1,\label{W:rec4}
\end{gather}
where
\begin{gather*}
\delta_n = -\tfrac12k^2 n(2n+2\nu+1),\\
\epsilon_n^{(2)} = \tfrac12k^2\nu(\nu+1)+\big(1-\tfrac12k^2\big)(2n+\nu+1)^2+\tfrac14 k^2\eta_2(4n+2\nu+3)\\
\hphantom{\epsilon_n^{(2)}}{} = \tfrac12 k^2\nu(\nu+1)+k'\big(2n+\tfrac32+\nu\big)+\big(1-\tfrac12k^2\big)\big(\tfrac14+\big(2n+\tfrac32+\nu\big)^2\big).
\end{gather*}

For the second Lam\'e--Wangerin problem we have results parallel to Theorem~\ref{W:t1}, Lem\-ma~\ref{W:l1} and Theorem~\ref{W:t2}. We denote the eigenvalues of the second Lam\'e--Wangerin problem by~$\HH_m(\nu,k)$, and the corresponding Lam\'e--Wangerin eigenfunctions by $w^{(2)}_m(z,\nu,k)$.

\section{Analytic continuation of Lam\'e--Wangerin functions}\label{AC}

In the previous section Lam\'e--Wangerin functions were defined on the interval $({\rm i}K',2K+{\rm i}K')$. We analytically continue these functions to the strip $0\le\Im z< K'$ as follows. Using~\eqref{W:etaz} and~\eqref{W:expansion1} a Lam\'e--Wangerin function $w^{(1)}$ of the first kind can be written as
\begin{gather}\label{AC:exp1}
w^{(1)}(z)={\rm e}^{-{\rm i}(\nu+1)(\frac12\pi-\am z)} \sum_{n=0}^\infty c_n (\sn z-{\rm i}\cn z)^{2n} .
\end{gather}
Since the power series $\sum c_n\eta^n$ has radius larger than $1$ and $|\eta|\le1$ for $0\le \Im z<K'$, the expansion~\eqref{AC:exp1} converges in the strip $0\le \Im z<K'$.

If $0<\eta<\eta_1$ we have
\begin{gather}\label{AC:id1}
\frac14 k^2\big(\eta^{1/2}+\eta^{-1/2}\big)^2-1= \frac14k^2\eta^{-1}(\eta_1-\eta)(\eta_2-\eta) .
\end{gather}
If $z$ is on the segment $({\rm i}K',K+{\rm i}K')$ and $\eta$ is given by \eqref{W:etaz} then \eqref{AC:id1} implies
\begin{gather}\label{AC:dn}
 {\rm i}\dn z=\frac12 k \eta^{-1/2}(\eta_1-\eta)^{1/2}(\eta_2-\eta)^{1/2} .
\end{gather}
Therefore, \eqref{W:expansion3} implies
\begin{gather}\label{AC:exp2}
w^{(2)}(z)=2{\rm i}k^{-1}{\rm e}^{-{\rm i}(\nu+2)(\frac12\pi-\am z)}\dn z\sum_{n=0}^\infty d_n(\sn z-{\rm i}\cn z)^{2n} .
\end{gather}
Again, this expansion is convergent in the strip $0\le \Im z<K'$.

In order to deal with expansions \eqref{W:expansion2} and \eqref{W:expansion4} we introduce the function
\begin{gather}\label{AC:I1}
 I_1(z):= (\dn z+\cn z)^{1/2}
\end{gather}
also appearing in \cite{I2}. This function is analytic in the strip $-K'<\Im z<K'$ when the branch of the root is chosen as follows. The function $\dn z+\cn z$ does not assume negative values or zero in the rectangle $-2K<\Re z<2K$, $-K'<\Im z<K'$. We choose the principal branch of the root in~\eqref{AC:I1} in this rectangle. We choose positive imaginary roots on the segments $(-2K,-2K+{\rm i}K')$ and $(2K-{\rm i}K',2K)$. For other~$z$, $I_1(z)$ is determined by $I_1(z+4K)=-I_1(z)$. A second function is defined by
 \begin{gather*}
I_2(z):=-I_1(z+2K)=-(\dn z-\cn z)^{1/2}.
\end{gather*}

For $0<\eta<\eta_1$ we have the identity
\begin{gather}\label{AC:id2}
(1-\eta+q)^{1/2}+ (1-\eta-q)^{1/2} =2^{1/2}(1-k')^{1/2}(\eta_2-\eta)^{1/2} ,
\end{gather}
where
\begin{gather*} q=k(\eta_1-\eta)^{1/2}(\eta_2-\eta)^{1/2}\end{gather*}
and all roots denote positive roots of positive numbers. For $z$ between ${\rm i}K'$ and $K+{\rm i}K'$ we have
\begin{gather*} {\rm i}\cn z=\frac12\big(\eta^{-1/2}-\eta^{1/2}\big) .\end{gather*}
Therefore, it follows from \eqref{AC:dn} and \eqref{AC:id2} that the analytic continuation of $J_1=\eta^{-1/4}(\eta_2-\eta)^{1/2}$ to the strip
$-K'<\Im z< K'$ is given by
\begin{gather}\label{AC:J1}
 (1-k')^{1/2}J_1(z)= {\rm e}^{{\rm i}\frac14\pi} I_1(z)+{\rm e}^{-{\rm i}\frac14\pi} I_2(z) .
\end{gather}
Therefore, the analytic continuation of the Lam\'e--Wangerin function $w^{(1)}$ of the first kind given by~\eqref{W:expansion2} to the strip $0\le \Im z<K'$ is
\begin{gather*}
w^{(1)}(z)={\rm e}^{-{\rm i}(\nu+\frac32)(\frac12\pi-\am z)}J_1(z)\sum_{n=0}^\infty a_n (\sn z-{\rm i}\cn z)^{2n} .
\end{gather*}
Similarly, for $0<\eta<\eta_1$ we have
\begin{gather*}
(1-\eta+q)^{1/2}- (1-\eta-q)^{1/2} =2^{1/2}(1+k')^{1/2}(\eta_1-\eta)^{1/2}.
\end{gather*}
It follows that the analytic continuation of the function $J_2=\eta^{-1/4}(\eta_1-\eta)^{1/2}$ to the strip $-K'<\Im z<K'$ is given by
\begin{gather}\label{AC:J2}
 (1+k')^{1/2}J_2(z)= {\rm e}^{{\rm i}\frac14\pi} I_1(z)-{\rm e}^{-{\rm i}\frac14\pi} I_2(z) .
\end{gather}
If a Lam\'e--Wangerin function $w^{(2)}$ of the second kind is given by~\eqref{W:expansion4} then its analytic conti\-nuation is
\begin{gather*}
w^{(2)}(z)={\rm e}^{-{\rm i}(\nu+\frac32)(\frac12\pi-\am z)}J_2(z)\sum_{n=0}^\infty b_n (\sn z-{\rm i}\cn z)^{2n} .
\end{gather*}

\section{Comparison of eigenvalues}\label{C}

Every Lam\'e--Wangerin function is also a Floquet eigenfunction.

\begin{Lemma}\label{C:l1} Let $\nu\in\R$ and $0<k<1$. A Lam\'e--Wangerin function $w^{(1)}(z)$ of the first kind satisfies
\begin{gather*} w^{(1)}(z+2K)={\rm e}^{{\rm i}(\nu+1)\pi} w^{(1)}(z).\end{gather*}
A Lam\'e--Wangerin function $w^{(2)}(z)$ of the second kind satisfies
\begin{gather*} w^{(2)}(z+2K)={\rm e}^{{\rm i}\nu\pi} w^{(2)}(z).\end{gather*}
\end{Lemma}
\begin{proof}This follows from \eqref{AC:exp1} and \eqref{AC:exp2}.
\end{proof}

If $\mu+\nu$ or $\mu-\nu$ is an integer then the Floquet eigenvalues $h_m(\mu,\nu)$ can be expressed in terms of Lam\'e--Wangerin eigenvalues~$H_m^{(j)}(\nu)$. The properties~\eqref{F:rules} show that it is sufficient to
consider the case $\mu=\nu+1$. Then we have the following result.

\begin{Theorem}\label{C:t1} Let $0<k<1$, $\nu\in\R$, $p\in\N_0$ and $p-1<|\nu|\le p$.
\begin{enumerate}\itemsep=0pt
\item[$(a)$] If $\nu\ge 0$ then
\begin{gather*}
h_m(\nu+1,\nu) = \HH_m(-\nu-1),\qquad m=0,1,\dots,p-1,\\
h_{p+2i+1}(\nu+1,\nu) = \HH_{p+i}(-\nu-1),\qquad i\ge 0,\\
h_{p+2i}(\nu+1,\nu) = \H_i(\nu),\qquad i\ge 0.
\end{gather*}
\item[$(b)$] If $\nu<0$ then
\begin{gather*}
h_m(\nu+1,\nu) = \H_m(\nu),\qquad m=0,1,\dots,p-1,\\
h_{p+2i+1}(\nu+1,\nu) = \H_{p+i}(\nu),\qquad i\ge 0,\\
h_{p+2i}(\nu+1,\nu) = \HH_i(-\nu-1),\qquad i\ge 0.
\end{gather*}
\item[$(c)$] If $\nu$ is an integer then
\begin{gather*} h_0(\nu+1,\nu)<h_1(\nu+1,\nu)<\dots<h_p(\nu+1,\nu) ,\end{gather*}
and, for $i\ge 0$,
\begin{gather*} h_{p+2i}(\nu+1,\nu)=h_{p+2i+1}(\nu+1,\nu)< h_{p+2i+2}(\nu+1,\nu)=h_{p+2i+3}(\nu+1,\nu)<\cdots.\end{gather*}
\end{enumerate}
\end{Theorem}
\begin{proof} (a), (b) Suppose first that $\nu$ is not an integer. Then the eigenvalues $f_m(k):=h_m(\nu+1,\nu,k)$ form a strictly increasing sequence. By Lemma~\ref{C:l1}, each eigenvalue $g_m(k):=\H_m(\nu,k)$ and $G_m(k):=\HH_m(-\nu-1,k)$ is among the $f$-eigenvalues. This is also true for $k=0$. The sequence $\{f_m(0)\}_{m=0}^\infty$ is given by the sequence $\big\{(2n+\nu+1)^2\big\}_{n\in\Z}$ when arranged in increasing order, $\{g_m(0)\}$ is given by $\{(2n+\nu+1)^2\}_{n=0}^\infty$ in increasing order and $\{G_m(0)\}$ is given by $\big\{(2n+\nu+1)^2\big\}_{n=-\infty}^{-1}$ in increasing order. Because of continuity of the functions $f_m$, $g_m$, $G_m$ (Theorems~\ref{F:t1} and~\ref{W:t2}) the order of these eigenvalues is the same for all $k\in[0,1)$. An analysis of the order at $k=0$ proves (a) and (b) for noninteger~$\nu$. The result extends to integer $\nu$ by continuity.

(c) Let $\nu$ be an integer. Then we apply~(a) or~(b) to $\nu+\epsilon$ in place of $\nu$ and take limits $\epsilon\to 0^{\pm}$. This proves~(c).
\end{proof}

We now compare the eigenvalues $H_m^{(j)}(\nu)$ with $H_m^{(j)}(-\nu-1)$.

\begin{Theorem}\label{C:t2}Let $p\in\N_0$, $0<k<1$. Let either $H=\H$ or $H=\HH$.
\begin{enumerate}\itemsep=0pt
\item[$(a)$] If $-p-\tfrac32<\nu<-p-\tfrac12$ then
\begin{gather*} H_p(\nu)<H_0(-\nu-1)<H_{p+1}(\nu)<H_1(-\nu-1)<H_{p+2}(\nu)<\cdots. \end{gather*}
\item[$(b)$] If $\nu=-p-\tfrac12$ then
\begin{gather*} H_{p-1}(\nu)<H_0(-\nu-1)=H_p(\nu)<H_1(-\nu-1)=H_{p+1}(\nu)<\cdots. \end{gather*}
\end{enumerate}
\end{Theorem}
\begin{proof} We consider the eigenvalues $H_m=\H_m$. The proof for $H_m=\HH_m$ is similar.

(a) Let $-p-\tfrac32<\nu<-p-\tfrac12$. The eigenvalues $H_m(\nu)$, $m\ge 0$, are pairwise distinct, and the eigenvalues $H_\ell(-\nu-1)$, $\ell\ge 0$, are pairwise distinct. The eigenvalues $H_m(\nu)$ are also distinct from $H_\ell(-\nu-1)$ because the corresponding eigenfunctions are linearly independent. Therefore, using continuity of the functions $k\mapsto H_m(\nu,k)$ (Theorem~\ref{W:t2}) the order of the eigenvalues $H_m(\nu,k)$, $H_\ell(-\nu-1,k)$ must be the same for all $k\in[0,1)$. The sequence $\{H_m(\nu,0)\}_{m=0}^\infty$ agrees with $\big\{(2n+\nu+1)^2\big\}_{n=0}^\infty$ after the latter sequence is arranged in increasing order. Similarly, the sequence $\{H_m(-\nu-1,0)\}_{m=0}^\infty$ agrees with $\big\{(2n-\nu)^2\big\}_{n=0}^\infty$ arranged in increasing order. Analysis of this order at $k=0$ implies~(a).

(b) If $p=0$ then $\nu=-\frac12$ and (b) is trivially true because $\nu=-\nu-1$. For $p\ge 1$ statement (b) follows from continuity of the functions $\nu\mapsto H(\nu,k)$ and taking one-sided limits as $\nu\to -p-\frac12$ in~(a).
\end{proof}

We now compare the eigenvalues $\H_m(\nu)$ with $\HH_m(\nu)$.

\begin{Theorem}\label{C:t3}Let $\nu\in\R$, $0<k<1$ and $H_m^{(j)}:=H_m^{(j)}(\nu,k)$.
\begin{enumerate}\itemsep=0pt
\item[$(a)$] If $\nu>-\frac32$ then
\begin{gather*} \H_0<\HH_0<\H_1<\HH_1<\H_2<\HH_2<\cdots.\end{gather*}
\item[$(b)$] If $-p-\frac32<\nu<-p-\frac12$ with $p\in\N$ then
\begin{gather*}\left\{\begin{matrix} \H_0 \\ \HH_0\end{matrix}\right\}< \dots< \left\{\begin{matrix} \H_{p-1} \\ \HH_{p-1}\end{matrix}\right\}< \H_p<\HH_p<\H_{p+1}<\HH_{p+1}<\cdots,\end{gather*}
where, for $m=0,1,\dots,p-1$,
\begin{gather*} \H_m<\HH_m \quad \text{if $m+p$ is even},\qquad \H_m>\HH_m \quad \text{if $m+p$ is odd} .\end{gather*}
\item[$(c)$] If $\nu=-p-\tfrac12$ with $p\in\N$ then
\begin{gather*} \H_0=\HH_0<\H_1=\HH_2<\dots < \H_{p-1}=\HH_{p-1}<\H_p<\HH_p<\cdots .\end{gather*}
\end{enumerate}
\end{Theorem}
\begin{proof} (a) Let $\nu>-\frac32$. Then the eigenfunctions of the two Lam\'e--Wangerin problems are constant multiples of solution~\eqref{W:form} with $q_0\ne 0$. Therefore, the eigenvalues of the two problems are mutually distinct. By continuity of the functions $k\mapsto H_m^{(j)}(\nu,k)$ (Theorem~\ref{W:t2}) the order of the eigenvalues $H_m^{(j)}$ must be the same for all $k\in[0,1)$. Now
\begin{gather*} \H_m(\nu,0)=(2m+\nu+1)^2,\qquad \HH_m(\nu,0)=(2m+\nu+2)^2, \end{gather*}
which implies (a).

(b) Let $-p-\frac32<\nu<-p-\frac12$ with $p\in\N$. Again, the eigenvalues of the two Lam\'e--Wangerin problems are mutually distinct, and the order of these eigenvalues must be the same for all $k\in[0,1)$. The sequence $\big\{H_m^{(j)}(\nu,0)\big\}_{m=0}^\infty$ is the same as $\{(2n+\nu+j)^2\}_{n=0}^\infty$ but the latter one has to be ordered increasingly. An analysis of the order leads to the arrangement stated in~(b).

(c) Let $\nu=-p-\tfrac12$ with $p\in\N$. Continuity of the functions $\nu\mapsto H_m^{(j)}(\nu,k)$ and part (b) show that $\H_m=\HH_m$ for $m=0,1,\dots,p-1$. We know from Theorem~\ref{C:t2}(b) that
\begin{gather*} H^{(j)}_{m+p}(\nu)=H_m^{(j)}(-\nu-1) , \qquad m\ge 0.\end{gather*}
Since $-\nu-1>-\frac32$ the rest of statement (c) follows from part (a).
\end{proof}

\section{Algebraic Lam\'e functions}\label{AL}
If $\nu+\frac12$ is a nonzero integer then Lam\'e's differential equation \eqref{I:lame} has solutions in finite terms which are usually called algebraic Lam\'e functions. These solutions were investigated in~\cite{Erd1,F,I2,L}. We obtain these functions as follows.

Let $\nu=-p-\tfrac12$ with $p\in\N$. For $j=1,2$ we introduce the symmetric tridiagonal $p$ by $p$ matrices
\begin{gather*} S_p^{(j)}=
\begin{pmatrix} \epsilon_0^{(j)} & \delta_1 & 0 & & \\
\delta_1 & \epsilon^{(j)}_1 & \ddots & \ddots & \\
0 & \ddots & \ddots & \ddots & 0\\
 &\ddots &\ddots & \epsilon^{(j)}_{p-2} & \delta_{p-1} \\
 & & 0 & \delta_{p-1} & \epsilon^{(j)}_{p-1}
\end{pmatrix},
\end{gather*}
where
\begin{gather*}
\epsilon_n^{(j)}=\tfrac12k^2\big(p^2-\tfrac14\big)+(-1)^j k'(2n+1-p)+\big(1-\tfrac12k^2\big)\big(\tfrac14+(2n+1-p)^2\big),\\
\delta_n=k^2n(p-n) .
\end{gather*}
The coefficient $\delta_p$ vanishes in \eqref{W:rec2}, \eqref{W:rec4}. Therefore, if $(a_0,a_1,\dots,a_{p-1})^t$ is an eigenvector of~$S^{(j)}_p$ and $a_n:=0$ for $n\ge p$ then \eqref{W:expansion2}, \eqref{W:expansion4} are Lam\'e--Wangerin functions of the first and second kind, respectively.

We note that $S_p^{(1)}$ is the mirror image of $S_p^{(2)}$ with respect to the anti-diagonal, that is, we have
\begin{gather*} \epsilon^{(1)}_{n}=\epsilon^{(2)}_{p-1-n},\qquad \delta_n=\delta_{p-n} .\end{gather*}
It follows that $S_p^{(1)}$ and $S_p^{(2)}$ have the same eigenvalues and the corresponding eigenvectors are inverse to each other, that is, if $(a_0,a_1,{\dots},a_{p-1})^t$ is an eigenvector for $S_p^{(1)}$ then
$(a_{p-1},a_{p-2},{\dots},a_0)^t$ is an eigenvector for $S_p^{(2)}$ belonging to the same eigenvalue. According to Theorem~\ref{C:t3}(c), the common eigenvalues of $S_p^{(j)}$ are
\begin{gather*} H_m^{(1)}(-p-\tfrac12)=H_m^{(2)}\big({-}p-\tfrac12\big),\qquad m=0,1,\dots,p-1 .\end{gather*}

If $(a_0,a_1,\dots,a_{p-1})^t$ is a (real) eigenvector of $S_p^{(1)}$ then
\begin{gather}\label{AL:alg1}
 \w=\eta^{-\frac12 p+\frac14}(\eta_2-\eta)^{1/2}\sum_{n=0}^{p-1} a_n \eta^n,\\
 \ww=\eta^{-\frac12 p+\frac14}(\eta_1-\eta)^{1/2}\sum_{n=0}^{p-1} a_{p-n-1}\eta^n \label{AL:alg2}
\end{gather}
are solutions of \eqref{W:lame2}. These are algebraic Lam\'e functions expressed in the variable $\eta$. We note that the functions $\w$ and $\ww$ are essentially Heun polynomials. For if we set $w=\eta^{-\frac12 p+\frac14}(\eta_j-\eta)^{1/2}v(s)$ and $\eta=\eta_1 s$, then we obtain the Heun equation for $v(s)$ and $\sum\limits_{n=0}^{p-1} a_n (\eta_1 s)^n$ is a~Heun polynomial.

If we substitute \eqref{W:etaz} in \eqref{AL:alg1}, \eqref{AL:alg2} and use the functions $J_1(z)$, $J_2(z)$ defined in \eqref{AC:J1}, \eqref{AC:J2} we obtain
\begin{gather*}
\w(z)= J_1(z) \sum_{n=0}^{p-1} a_n(\sn z-{\rm i}\cn z)^{2n-p+1},\\
\ww(z)= J_2(z)\sum_{n=0}^{p-1}a_n (\sn z+{\rm i}\cn z)^{2n-p+1}.
\end{gather*}
We know from Lemma \ref{C:l1} that
\begin{gather*}
\w(z+2K)= {\rm i}(-1)^p \w(z),\qquad
\ww(z+2K)= {\rm i}(-1)^{p+1} \ww(z).
\end{gather*}
Moreover, we have
\begin{gather*} (1+k')^{1/2}\overline{\w(\bar z)}=-{\rm i}(1-k')^{1/2} \ww(z), \end{gather*}
and, for $x\in\R$,
\begin{gather*}
\overline{\w(x)}= \w(2K-x),\qquad
\overline{\ww(x)}= -\ww(2K-x),
\end{gather*}
which shows that the real part of $\w(x)$ is a function even with respect to $x=K$ while the imaginary part of $\w(x)$ is odd with respect to $x=K$.

One should notice that if $\nu=-p-\frac12$ and $h$ is an eigenvalue of the matrix $S_p^{(j)}$ then all (nontrivial) solutions of Lam\'e's equation qualify as ``algebraic Lam\'e functions''. We picked the basis of two solutions, one even and one odd with respect to $z=K+{\rm i}K'$. Ince \cite{I2} considered the basis of even and odd solutions (with respect to $z=0$) while Erd\'elyi~\cite{Erd1} has the basis of even or odd solutions with respect to $z=K$.

In the simplest case $\nu=-\frac32$ we have
\begin{gather*} H_0^{(1)}\big({-}\tfrac32\big)=H_0^{(2)}\big({-}\tfrac32\big)=\tfrac14\big(1+k^2\big) \end{gather*}
and
\begin{gather*} w_0^{(1)}(z)=J_1(z),\qquad w_0^{(2)}(z)=J_2(z) .\end{gather*}
If $\nu=-\frac52$ then
\begin{gather*}
 H_0^{(1)}\big({-}\tfrac52\big)=H_0^{(2)}\big({-}\tfrac52\big)=\tfrac54\big(1+k^2\big)-\big(1-k^2+k^4\big)^{1/2},\\
 H_1^{(1)}\big({-}\tfrac52\big)=H_1^{(2)}\big({-}\tfrac52\big)=\tfrac54\big(1+k^2\big)+\big(1-k^2+k^4\big)^{1/2}.
\end{gather*}
If we choose $a_0=-k^2$, $a_1=\frac34 k^2+\frac94 -\frac12k^2\eta_1- H_m^{(1)}\big({-}\frac52\big)$, $m=0,1$, then
\begin{gather*}
 w_m^{(1)}(z) = J_1(z)(a_0(\sn z+{\rm i}\cn z)+a_1(\sn z-{\rm i}\cn z)),\\
 w_m^{(2)}(z) = J_2(z)(a_0(\sn z-{\rm i}\cn z)+a_1(\sn z+{\rm i}\cn z)).
\end{gather*}

\section{Lam\'e polynomials}\label{LP}
Let $\nu=-p-1$ with $p\in\N_0$. It is well-known \cite[Chapter~IX]{A} that there are $2p+1$ distinct values of $h$ for which~\eqref{I:lame} admits nontrivial solutions which are polynomials in $\cn z$, $\sn z$, $\dn z$. In our notation these values of $h$ are
\begin{gather}\label{LP:ev1}
 H_m^{(1)}(-p-1),\qquad m=0,1,\dots,p
\end{gather}
and
\begin{gather}\label{LP:ev2}
H_m^{(2)}(-p-1),\qquad m=0,1,\dots,p-1 .
\end{gather}
Since $\alpha_{p+1}=0$ in \eqref{W:rec1}, the numbers \eqref{LP:ev1} are the eigenvalues of the $p+1$ by $p+1$ tridiagonal matrix
\begin{gather*} T_{p+1}^{(1)}=
\begin{pmatrix} \beta_0 & \gamma_1 & 0 & & \\
\alpha_1 & \beta_1 & \ddots & \ddots & \\
0 & \ddots & \ddots & \ddots & 0\\
 &\ddots &\ddots & \beta_{p-1} & \gamma_p \\
 & & 0 & \alpha_p & \beta_p
\end{pmatrix},
\end{gather*}
where
\begin{gather*}
\alpha_n= \tfrac12 k^2(p+1-n)(2n-1),\\
\beta_n= \tfrac12k^2 p(p+1)+\big(1-\tfrac12k^2\big)(2n-p)^2,\\
\gamma_n= \tfrac12k^2 (2p+1-2n)n .
\end{gather*}
If $(c_0,c_1,\dots,c_p)^t$ is an eigenvector of $T_{p+1}^{(1)}$ then
\begin{gather*} w=\eta^{-p/2}\sum_{n=0}^p c_n \eta^n \end{gather*}
is a solution of \eqref{W:lame1}.
After substituting \eqref{W:etaz} we obtain
\begin{gather*} w=\sum_{n=0}^p c_n(\sn z-{\rm i}\cn z)^{2n-p} .\end{gather*}
Indeed, since $(\sn z-{\rm i}\cn z)^{-1}=\sn z+{\rm i} \cn z$, these solutions are polynomials in $\cn z$, $\sn z$. The matrix $T_{p+1}^{(1)}$ has the symmetries
\begin{gather*} \alpha_n=\gamma_{p-n+1},\qquad \beta_n=\beta_{p-n} .\end{gather*}
Therefore, the space of symmetric vectors $\{c_n\}_{n=0}^p$ ($c_n=c_{p-n}$, $n=0,1,\dots,p$), as well as the space of antisymmetric vectors is invariant under $T_{p+1}^{(1)}$. Thus eigenvectors of $T_{p+1}^{(1)}$ will lie in one of these invariant subspaces.

If $p$ is even we find $\tfrac12 p+1$ Lam\'e polynomials of the form $P\big(\sn^2 z\big)$ where $P$ is a polynomial of degree $\tfrac12p$ if we use symmetric eigenvectors, and $\tfrac12 p$ Lam\'e polynomials of the form $\cn z\sn z P\big(\sn^2 z\big)$ where $P$ is a polynomial of degree $\tfrac12 p-1$ if we use antisymmetric eigenvectors. If $p$ is odd we find $\tfrac12 (p+1)$ Lam\'e polynomials of the form $\sn z P\big(\sn^2 z\big)$ where $P$ is a~polynomial of degree $\tfrac12(p-1)$ if we use symmetric eigenvectors, and $\tfrac12 (p+1)$ Lam\'e polynomials of the form~$\cn z P\big(\sn^2 z\big)$ where~$P$ is a polynomial of degree~$\tfrac12(p-1)$ if we use antisymmetric eigenvectors.

Similarly, Lam\'e--Wangerin functions of the second kind belonging to the eigenvalues \eqref{LP:ev2} are Lam\'e polynomials that have the factor~$\dn z$.

\section{Zeros of Lam\'e--Wangerin functions}\label{Z}

We first determine the number of zeros of Lam\'e--Wangerin functions $w_m^{(j)}(z)$ in the open interval $({\rm i}K',K+{\rm i}K')$.
\begin{Theorem}\label{Z:t1} Let $j=1,2$, $m\in\N_0$, $\nu\in\R$, $k\in(0,1)$.
\begin{enumerate}\itemsep=0pt
\item[$(a)$] If $\nu>-\frac32$ then $w_m^{(j)}$ has exactly $m$ zeros in $({\rm i}K',K+{\rm i}K')$.
\item[$(b)$] If $-p-\tfrac32<\nu\le -p-\tfrac12$, $p\in \N$, then $w_m^{(j)}$ has exactly $\max\{0,m-p\}$ zeros in $({\rm i}K',K+{\rm i}K')$.
\end{enumerate}
\end{Theorem}
\begin{proof} Consider $j=1$. The proof for $j=2$ similar. Let $P$ be the set of all real numbers different from $-p-\frac12$ for all $p\in\N$. For $h\in\R$, $\nu\in P$, let $w(z,h,\nu)$ be the solution of \eqref{I:lame} given locally at $z={\rm i}K'$ by~\eqref{W:form} with $q_0=1$. Then one can show that $(z-{\rm i}K')^{-\nu-1}w(z,h,\nu)$ is continuous for $z$ in $[{\rm i}K',K+{\rm i}K']$, $h\in\R$, $\nu\in P$ ($k$~fixed). If we set $w_m(z,\nu)=w\big(z,\H_m(\nu,k),\nu\big)$ then $(z-{\rm i}K')^{-\nu-1} w_m(z,\nu)$ is continuous for $z\in[{\rm i}K',K+{\rm i}K']$ and $\nu\in P$. This implies that the number of zeros of $w_m(\cdot,\nu)$ in $({\rm i}K',K+{\rm i}K')$ is finite and it is locally constant as a function of~$\nu$.

(a) follows by considering $\nu=0$:
\begin{gather}\label{Z:nu0}
 w_m^{(1)}(z,0,k)=(-1)^m \sin\left((2m+1)\frac{\pi}{2K}(z-{\rm i}K')\right).
\end{gather}

(b) Suppose $-p-\tfrac32<\nu\le -p-\tfrac12$ with $p\in \N$. Let $m=0,1,\dots,p$. By Theorem \ref{C:t2}(a), we have $H_m(\nu)\le H_0(-\nu-1)$. By (a), $w_0(\cdot,-\nu-1)$ has no zeros in $({\rm i}K',K+{\rm i}K')$. Therefore, by Sturm comparison, $w_m(\cdot,\nu)$ also has no zeros in this interval.

Now consider $m>p$. If $\nu=-p-\frac12$ then, by Theorem \ref{C:t2}(b), $H_m(\nu)=H_{m-p}(-\nu-1)$. Therefore, by (a), $w_m(\cdot, \nu)$ has $m-p$ zeros in $({\rm i}K',K+{\rm i}K')$. If $-p-\frac32<\nu<p-\frac12$ then, by Theorem~\ref{C:t2}(a), $H_m(\nu)<H_{m-p}(-\nu-1)$. Therefore, $w_m(\cdot,\nu)$ can have at most $m-p$ zeros in $({\rm i}K',K+{\rm i}K')$. If $\nu=-p-\frac12$ we just showed that there are $m-p$ zeros. By continuity, there are
exactly $m-p$ zeros. For the latter step Lam\'e--Wangerin functions should be normalized by the initial conditions $w(K+{\rm i}K')=1$, $w'(K+{\rm i}K')=0$.
\end{proof}

Now we look for zeros of Lam\'e--Wangerin functions in the strip $0\le \Im z<K'$.

\begin{Lemma}\label{Z:l1}A Lam\'e--Wangerin function which is not a Lam\'e polynomial has no zeros on the real axis.
\end{Lemma}
\begin{proof}If $\mu$ is not an integer then a nontrivial Floquet solution $w(z)$, $z\in\R$, of \eqref{I:lame} with $w(z+2K)={\rm e}^{{\rm i}\mu \pi}w(z)$ does not have zeros on the real axis. This is because the conjugate of~$w(z)$ is a~Floquet solution with conjugate multiplier ${\rm e}^{-{\rm i}\mu \pi}$, and ${\rm e}^{{\rm i}\mu \pi}$, ${\rm e}^{-{\rm i}\mu \pi}$ are distinct. So~$w(z)$ and its conjugate function are linearly independent. It follows from Lemma~\ref{C:l1} that Lam\'e--Wangerin functions have no zeros on the real axis if~$\nu$ is not an integer.

Suppose that $\nu$ is an integer, and $w(z)$ is a Lam\'e--Wangerin function belonging to the eigenvalue $\H_m(\nu)$. Suppose that $w(z_0)=0$.\ with $z_0\in\R$. Using \eqref{W:expansion1} and the substitution \eqref{F:am} we have
\begin{gather*} w(t)=\sum_{n=0}^\infty c_n {\rm e}^{-{\rm i}t(2n+\nu+1)} \end{gather*}
and this function has a zero at $t_0\in\R$. The coefficients $c_n$ are real so the functions
\begin{gather}
\Re w(t)= \sum_{n=0}^\infty c_n \cos(2n+\nu+1) t ,\qquad t\in\R\label{Z:cos}\\
\Im w(t)= -\sum_{n=0}^\infty c_n \sin(2n+\nu+1)t,\qquad t\in\R\label{Z:sin}
\end{gather}
both vanish at $t=t_0$. The functions~\eqref{Z:cos}, \eqref{Z:sin} are both solutions of the differential equa\-tion~\eqref{F:lame2} with the same values for $h$ and $\nu$. Since they have a common zero these solutions
must be linearly dependent. Now~$\Re w(t)$ is an even function and $\Im w(t)$ is odd. So one of the functions~$\Re w(t)$,~$\Im w(t)$ must vanish identically. This implies that $c_n=0$ for large enough $n$ and so $w(z)$ is a Lam\'e polynomial. The proof is similar for Lam\'e--Wangerin function of the second kind.
\end{proof}

According to \eqref{W:expansion2} we write a Lam\'e--Wangerin function of the first kind as
\begin{gather*} w_m^{(1)} =\eta^{(\nu+1)/2}(\eta_2-\eta)^{1/2} v^{(1)}_m(\eta,\nu,k) ,\end{gather*}
where
\begin{gather*} v_m^{(1)}(\eta,\nu,k)=\sum_{n=0}^\infty a_n\eta^n \end{gather*}
is given by a power series with radius $\ge\eta_2>1$. Similarly, we write a Lam\'e--Wangerin function of the second kind as
\begin{gather*} w_m^{(2)} =\eta^{(\nu+1)/2}(\eta_1-\eta)^{1/2} v^{(2)}_m(\eta,\nu,k) .\end{gather*}

\begin{Theorem}\label{Z:t2} Let $m\in\N_0$, $\nu\in\R$, $k\in(0,1)$.
\begin{enumerate}\itemsep=0pt
\item[$(a)$] Suppose that $-m-\nu\not\in\N$, and choose $\ell\in\N_0$ such that $\H_m(\nu,0)=(2\ell+\nu+1)^2$; see Lemma~{\rm \ref{W:l1}}. Then $v_m^{(1)}(\cdot,\nu,k)$ has exactly $\ell$ zeros in the unit disk $|\eta|<1$ counted by multiplicity.
\item[$(b)$] Suppose that $-m-\nu-1\not\in\N$, and choose $\ell\in\N_0$ such that $\HH(\nu,0)=(2\ell+\nu+2)^2$. Then $v_m^{(2)}(\cdot,\nu,k)$ has exactly $\ell$ zeros in the unit disk $|\eta|<1$ counted by multiplicity.
\end{enumerate}
\end{Theorem}
\begin{proof} We prove only (a). The proof of (b) is similar. We normalize the Lam\'e--Wangerin functions $w_m(z)$ of the first kind by setting $w_m(K+{\rm i}K')=1$. Then $w_m(z,\nu,k)$ is the solution of~\eqref{I:lame} with $h=\H_m(\nu,k)$ determined by the initial conditions $w(K+{\rm i}K')=1$, $w'(K+{\rm i}K')=0$. By continuous parameter dependence of solutions of initial value problems of linear differential equations, and using Theorem~\ref{W:t2}, we obtain that $w_m(z,\nu,k)$ is a continuous function of $(z,\nu,k)$ for $z\in\R$, $\nu\in\R$, $k\in(0,1)$. Since $|\eta|=1$ is in correspondence with $z\in[0,2K)$, we see that $v_m(\eta,\nu,k)$ is a~continuous function of $|\eta|=1$, $\nu\in\R$, $k\in(0,1)$. We want to apply Rouch\'e's theorem to the homotopy $s\mapsto v_m(\eta,s\nu,k)$ for $s\in[0,1]$. If $v_m(\eta,s\nu,k)\ne 0$ on the unit circle $|\eta|=1$ for all $s\in[0,1]$, then $v_m(\cdot, s\nu,k)$ has the same number of zeros in $|\eta|<1$ for $s\in[0,1]$.

Suppose that $\nu>-m-1$. By Lemma \ref{Z:l1}, the function $v_m(\eta,s\nu,k)$ has no zeros on the unit circle $|\eta|=1$ for $0\le s\le 1$ and so the number of zeros of $v_m(\cdot,\nu,k)$ in the open unit disk agrees with that of $v_m(\cdot,0,k)$. It follows from~\eqref{Z:nu0} that the number of zeros of $v_m(\cdot,\nu,k)$ in the open unit disk is equal to~$m$. Under our assumption on $(\nu,m)$ we have $\ell=m$, so we obtain state\-ment~(a) for $\nu>-m-1$.

Now we assume that $-p-1<\nu<-p$ with $p\in\N$ and $m<p$. We use similar homotopies to show that the number of zeros of $v_m(\cdot,\nu,k)$ may depend on $p$ and $m$ but not on $\nu, k\in(0,1)$. So we consider $\nu=-p-\frac12$. Then $w$ is an algebraic Lam\'e function and $v_m(\eta)= \sum\limits_{n=0}^{p-1} a_n \eta^n$ is a~polynomial. Let $k_n\in(0,1)$ be a sequence converging to~$0$. Since the vector $(a_0,a_1,\dots,a_{p-1})^t$ is an eigenvector of the matrix $S_p^{(1)}$ from Section~\ref{AL} it is easy to see that when properly normalized the eigenvectors belonging to $k_n$ converge to the vector $(a_0,\dots,a_{p-1})^t$ with all components equal to $0$ except $a_\ell=1$. Therefore, under the new normalization $v_m\big(\eta,-p-\frac12,k_n\big)$ converges uniformly to $\eta^\ell$ as $n\to 0^+$. By Rouch\'e's theorem, we obtain the desired statement.

This completes the proof.
\end{proof}

Using the map \eqref{W:etaz} the unit disk $|\eta|<1$ can be related to a domain in the $z$-plane. Consider the rectangle
\begin{gather*} Q=\{z\in\C\colon 0<\Re z<2K,\, 0<\Im z<K'\} .\end{gather*}
The function $z\mapsto \eta$ is a conformal map from $Q$ onto the unit disk $|\eta|<1$ with a branch cut along the interval $(-1,\eta_1)$. If $z$ starts at $z=0$ and moves clockwise around the boundary of~$Q$, then $\eta$ starts at $\eta=-1$ and moves in the mathematically positive direction along the unit circle returning to $\eta=-1$ when $z=2K$. Then $\eta$ moves from $\eta=-1$ to $\eta=0$ when $z$ reaches $z={\rm i}K'$. Then $\eta$ moves to $\eta_1$ when $z=K+{\rm i}K'$ and returns to $\eta=0$, then to $\eta=-1$. It follows that the set
\begin{gather*} \tilde Q=\{z\colon 0\le \Re z\le K, \, 0<\Im z\le K'\}\cup \{z\colon K<\Re z<2 K, \, 0<\Im z<K'\}\end{gather*}
is mapped bijectively onto the open unit disk $|\eta|<1$.

Theorem \ref{Z:t2}(a) can be extended to include Lam\'e polynomials. If $\nu=-m-1,-m-2,\dots$ then let $\ell_1$ be the smallest nonnegative integer $n$ satisfying $\H_m(\nu,0)=(2n+\nu+1)^2$ and $\ell_2$ the largest such integer. Then $v_m$ has $\ell_1$ zeros in the open unit disk $|\eta|<1$ and $\ell_2$ zeros in the closed unit disk $|\eta|\le 1$. This follows from the known location of zeros of Lam\'e polynomials \cite[Section~9.4]{A}.
Similarly, Theorem~\ref{Z:t2}(b) can be extended.

\section[The limit $k\to 0$ of Lam\'e--Wangerin functions]{The limit $\boldsymbol{k\to 0}$ of Lam\'e--Wangerin functions}\label{L}
Substituting $u=\frac{2K}{\pi} s$ in \eqref{W:SL}, we obtain the differential equation
\begin{gather}\label{L:ode1}
\frac{{\rm d}^2w}{{\rm d}s^2}+\frac{4K^2}{\pi^2}\left(h-\nu(\nu+1)\sn^{-2} \left(\frac{2K}{\pi} s\right)\right) w =0,\qquad 0<s<\pi .
\end{gather}
In \eqref{L:ode1} we set $h=\H_m(\nu,k)$ and take $w=w_m^{(1)}(s,\nu,k)$ as the corresponding Lam\'e--Wangerin eigenfunction
normalized by the initial condition
\begin{gather*} w\left(\frac\pi2\right)=1,\qquad \frac{{\rm d}w}{{\rm d}s}\left(\frac\pi2\right)=0 .\end{gather*}
By Theorem \ref{W:t2}, $\H_m(\nu,k)\to \H_m(\nu,0)=(2\ell+\nu+1)^2$ as $k\to 0^+$, where $\ell\in\N_0$ is chosen according to Lemma~\ref{W:l1}. As $k\to 0^+$ we see that $w_m^{(1)}(s,\nu,k)$ converges to the solution $W^{(1)}_m(s,\nu)$ of the differential equation
\begin{gather}\label{L:ode2}
\frac{{\rm d}^2W}{{\rm d}s^2}+\left((2\ell+\nu+1)^2-\frac{\nu(\nu+1)}{\sin^2 s}\right) W =0
\end{gather}
satisfying the initial conditions
\begin{gather*} W\left(\frac\pi2\right)=1,\qquad \frac{{\rm d}W}{{\rm d}s}\left(\frac\pi2\right)=0 .\end{gather*}
The convergence
\begin{gather*} w_m^{(1)}(s,\nu,k)\to W_m^{(1)}(s,\nu)\qquad \text{as $k\to 0^+$} \end{gather*}
is uniform on compact subintervals of $(0,\pi)$. Differential equation \eqref{L:ode2} appears in the theory of Gegenbauer polynomials \cite[equation~(4.7.11)]{Sz}. We find that
\begin{gather*} W_m^{(1)}(s,\nu)=(\sin s)^{\nu+1} F\big({-}\ell,\ell+\nu+1;\tfrac12;\cos^2 s\big) ,\end{gather*}
where $F$ denotes the hypergeometric function. Equivalently, using Gegenbauer polyno\-mials $G_n^{(\lambda)}(x)$ we have \cite[equation~(4.7.30)]{Sz}
\begin{gather*} W_m^{(1)}(s,\nu)=(\sin s)^{\nu+1}(-1)^\ell \binom{\ell+\nu}{\ell}^{-1} G_{2\ell}^{(\nu+1)}(\cos s) .\end{gather*}
The binomial coefficient may vanish but the formula remains valid if we take limits $\nu\to \nu_0$ at exceptional values $\nu=\nu_0$.

Similarly, let $\ww_m(s,\nu,k)$ be the solution of \eqref{L:ode1} with $h=H_m^{(2)}(\nu,k)$ satisfying the initial conditions
\begin{gather*} w\left(\frac\pi2\right)=0,\qquad \frac{{\rm d}w}{{\rm d}s}\left(\frac\pi2\right)= 1 .\end{gather*}
We choose $\ell\in\N_0$ such that $\HH_m(\nu,0)=(2\ell+\nu+2)^2$. Then we obtain
\begin{gather*} w_m^{(2)}(s,\nu,k)\to W^{(2)}_m(s,\nu)=-(\sin s)^{\nu+1}\cos s F\big({-}\ell,\ell+\nu+2;\tfrac32;\cos^2 s\big) \end{gather*}
as $k\to0^+$ uniformly on compact subintervals of $(0,\pi)$. In terms of Gegenbauer polynomials we have
\begin{gather*} W^{(2)}_m(s,\nu)=(\sin s)^{\nu+1}(-1)^{\ell+1}\left( 2(\nu+1) \binom{\ell+\nu+1}{\ell}\right)^{-1} G_{2\ell+1}^{(\nu+1)}(\cos s) .\end{gather*}

\pdfbookmark[1]{References}{ref}
\LastPageEnding

\end{document}